\documentclass{amsart}
\usepackage{amsthm}
\usepackage{amssymb}
\usepackage{mathtools}
\usepackage{graphicx}
\usepackage{enumitem}
\usepackage[all]{xy}
\usepackage{tikz}
\usetikzlibrary{arrows}
\usepackage{caption}

\newenvironment{dedication}
  {
   \vspace*{.3cm}
   \itshape             
   \raggedleft          
  }
  {\par 
   \vspace*{.3cm} 
  }

\definecolor{lightgrey}{rgb}{.804,.804,.756}
\usepackage[pdfauthor={V. Lebed, L. Vendramin},
pdftitle={Yang--Baxter equation vs reflection equation},
colorlinks=false,linkbordercolor=lightgrey,citebordercolor=lightgrey,urlbordercolor=lightgrey]{hyperref}

\definecolor{myred}{rgb}{.545,0,0}
\definecolor{myblue}{rgb}{.024,.15,.645}    
\definecolor{mygreen}{rgb}{0,.455,0} 
\definecolor{myviolet}{rgb}{.45,.05,.545}       
\definecolor{lightgrey}{rgb}{.8,.8,.8}

\newcommand{\N}{\mathbb{N}}
\newcommand{\Z}{\mathbb{Z}}

\newcommand{\id}{\operatorname{Id}}
\newcommand{\Id}{\operatorname{Id}}

\newcommand{\Refl}{\operatorname{Refl}}

\newcommand{\Strange}{\operatorname{Strange}}

\newcommand\op{\mathrel{\triangleleft}}
\newcommand\wop{\mathrel{\widetilde{\triangleleft}}}
\newcommand\lop{\mathrel{\triangleright}}
\newcommand\wlop{\mathrel{\widetilde{\triangleright}}}
\newcommand\kop{\mathrel{\blacktriangleleft_k}}
\newcommand\klop{\mathrel{\blacktriangleright_k}}

\makeatletter
	\theoremstyle{plain}
\newtheorem{thm}{Theorem}[section]
\newtheorem{lem}[thm]{Lemma}
\newtheorem{cor}[thm]{Corollary}
\newtheorem{pro}[thm]{Proposition}
\newtheorem*{conjecture*}{Conjecture}

	\theoremstyle{definition}
\newtheorem{defn}[thm]{Definition}

\newtheorem{exa}[thm]{Example}
	\theoremstyle{remark}
\newtheorem{rem}[thm]{Remark}

\setlist{nolistsep}
\setenumerate{topsep=0pt}
\makeatother

\begin{document}

\title[]{Reflection equation as a tool for studying solutions to the Yang--Baxter equation}

\begin{abstract}
Given a right-non-degenerate set-theoretic solution $(X,r)$ to the Yang--Baxter equation, we construct a whole family of YBE solutions $r^{(k)}$ on $X$ indexed by  its reflections $k$ (i.e., solutions to the reflection equation for $r$). This family includes the original solution and the classical derived solution. All these solutions induce isomorphic actions of the braid group/monoid on $X^n$. The structure monoids 
of $r$ and $r^{(k)}$ are related by an explicit bijective $1$-cocycle-like map. We thus turn reflections into a tool for studying YBE solutions, rather than a side object of study. In a different direction, we study the reflection equation for non-degenerate involutive YBE solutions, show it to be equivalent to (any of the) three simpler relations, and deduce from the latter systematic ways of constructing new reflections.
\end{abstract}



\author{Victoria Lebed}
\address{ LMNO, Universit\'e de Caen--Normandie, BP 5186, 14032 Caen Cedex, France}
\email{lebed@unicaen.fr}
  
\author{Leandro Vendramin}
\address{Department of Mathematics, Vrije Universiteit Brussel, Pleinlaan 2, 1050 Brussel; and 
Departamento de Matem\'atica -- FCEN,
Universidad de Buenos Aires, Pab. I -- Ciudad Universitaria (1428),
Buenos Aires, Argentina}
\email{lvendramin@dm.uba.ar}

\maketitle

\begin{dedication}
To the memory of Patrick Dehornoy,\\ a connoisseur of strange structures  
\end{dedication} 

\section*{Introduction}

Let $X$ be a set. A map $r \colon X \times X \to X \times X$ will be called a \emph{solution} if it is a (set-theoretic) solution to the \emph{Yang--Baxter equation (YBE)} on $X^3$:
\begin{align}
r_1r_2r_1 &=r_2r_1r_2. \label{E:YBE}
\end{align}
Here and below the notation $\phi_i$ means the map $\phi$ applied starting from the $i$th position of some $X^{m}$. A map $k \colon X \to X$ will be called a \emph{reflection} for $(X,r)$ if it satisfies the \emph{reflection equation (RE)} on $X^2$:
\begin{align}
rk_2rk_2 &= k_2rk_2r. \label{E:RE}
\end{align}

The linear versions of these equations (that is, when $X$ is a vector space, the direct products $\times$ are replaced with the tensor products $\otimes$, and all maps are required to be linear) appeared in physics. In particular, in the study of $n$ particle scattering on a half-line, a solution $r$ represents a collision between two particles, and a reflection $k$ represents a collision between a particle and the wall delimiting the half-line. See Figs. \ref{P:Scattering}--\ref{P:Scattering2} for an illustration.

\begin{figure}[h]
\centering
\begin{tikzpicture}[scale=0.7]
\draw (0,0)--(1,2);
\draw (.7,0)--(0,2);
\draw (2.5,0)--(-1,2);
\draw [|->, red, thick]  (-1.8,0) -- (-1.8,2);
\node [red] at (-2.5,1){time};
\draw [myviolet, line width=.5mm]  (3,0) -- (3,2);
\node [myviolet] at (3,-.5){wall};
\draw [myblue, line width=.4mm]  (3,0) -- (-1.2,0);
\node [myblue] at (.8,-.5){half-line};
\node  at (4.5,1){$=$};
\node  at (5.5,1){ };
\end{tikzpicture}
\begin{tikzpicture}[scale=0.7]
\draw (0.2,0)--(1,2);
\draw (1,0)--(0.3,2);
\draw (1.8,0)--(-1.5,2);
\draw [myviolet, line width=.5mm]  (3,0) -- (3,2);
\draw [myblue, line width=.4mm]  (3,0) -- (-1.8,0);
\node [myviolet] at (3,-.7){ };
\end{tikzpicture}
\caption{The Yang--Baxter equation governing collisions between particles moving on a half-line}\label{P:Scattering}
\end{figure}
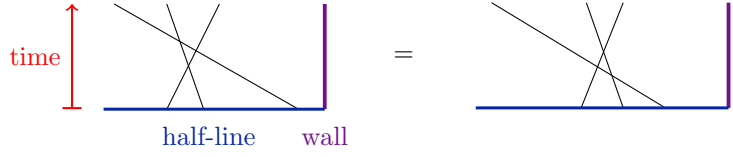

\begin{figure}[h]
\centering
\begin{tikzpicture}[scale=0.8]
\draw (0,0)--(2,1)--(0,2);
\draw (1,0)--(2,0.8)--(.5,2);
\draw [myviolet, line width=.5mm]  (2,0) -- (2,2);
\node  at (4,1){$=$};
\node  at (5,1){ };
\end{tikzpicture}
\begin{tikzpicture}[scale=0.8]
\draw (0,0)--(2,1)--(0,2);
\draw (.5,0)--(2,1.2)--(1,2);
\draw [myviolet, line width=.5mm]  (2,0) -- (2,2);
\end{tikzpicture}
\caption{The reflection equation governing collisions between particles and the wall}\label{P:Scattering2}
\end{figure}
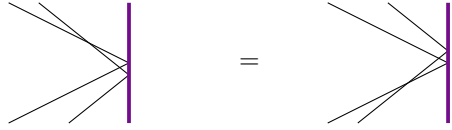

Since then the YBE has made its way into many areas of mathematics. The RE is far less familiar to mathematicians. An exception is the study of braids with a frozen strand, or, equivalently, the study of the motion of $n$ particles in an annulus (whereas usual braids describe the motion of $n$ particles in a disk). Figs. \ref{P:YBE}--\ref{P:RE} should make this connection clear. See \cite{Chow,Sossinsky,Schwiebert} for more detail.

The idea of studying set-theoretic rather than linear YBE solutions goes back to Drinfel$'$d. His goal was to restrict attention to this more tractable, but still extremely rich class of solutions. This class is currently an object of active mathematical research; see \cite{SS18,GI18,JKVA19,MBBER19,Rump20,CCS20,CJO20} and references therein for some recent advances. The interest in set-theoretic solutions to the RE is on the contrary rather new. It was started in \cite{MR3030177,MR3207925}, where some explicit examples were given. The only articles going urther in this direction we are aware of are  \cite{DC18,KuOka,K19,SVW18,DoiSmo}. The objective of all those contributions is to systematically construct and study reflections for different types of solutions. 

In the present paper, we are looking at reflections from a completely different angle. For us they are tools for studying the corresponding solution. Concretely, with any solution $r$ on $X$, written explicitly as
\begin{align}
r(a,b) = (\lambda_a(b),\rho_b(a)),\label{E:r_explicit}
\end{align}
 come the following invariants:
\begin{itemize}
\item the \emph{structure monoid} of $r$, given by the presentation
\begin{align}
M(X,r) &= \langle a\in X\mid ab=\lambda_a(b) \rho_b(a), \text{ for all } a,b\in X \rangle;\label{E:StrMonoid}
\end{align}
\item the \emph{structure shelf} of $r$:
\[a \op b = \rho_b \lambda_{\rho_a^{-1}(b)}(a);\]
\item the \emph{derived monoid} $A(X,r)$ of $r$, which is the structure monoid of the following \emph{derived solution} on $X$:
\begin{align}
r_{\op} \colon (a,b) \mapsto (b \op a,a);\label{E:DerSol}
\end{align}
\item the \emph{structure group} $G(X,r)$ and the \emph{derived group} of $r$, given by the same presentation as the monoids $M(X,r)$ and $A(X,r)$ respectively. 
\end{itemize}
The constructions of structure shelf and derived monoid/group are valid only for \emph{right-non-degenerate (RND)}\footnote{In \cite{LebVen} we used the term \emph{LND} instead; here we stick to the more popular convention.} solutions $r$, that is, their right components $\rho_b$ are bijective on $X$ for all $b \in X$. Similarly, a solution is called \emph{left-non-degenerate (LND)} if all the $\lambda_a$'s are bijective, and \emph{non-degenerate} if it is both RND and LND.

Recall that a \emph{shelf} is a set $X$ with a binary self-distributive operation $\op$, in the sense of 
\[(a \op b) \op c = (a \op c) \op (b \op c)\]
for all $a,b,c \in X$. For any shelf, the map \eqref{E:DerSol} is a solution.

Note that all these constructions are not only useful invariants of solutions, but also a rich source of monoids and groups with remarkable properties.

The structure and the derived monoids are related by a bijective monoid $1$-cocycle
\[\overline{J} \colon M(X,r) \to A(X,r).\]
Similarly, the structure and the derived groups are related by a bijective group $1$-cocycle. As a result, one can deduce various properties of the structure monoid and group of a solution from those of its structure shelf. This was used to understand the possible values of the size of the degree $2$ component of $M(X,r)$ in \cite{CJO_Minimality}; to construct a remarkable finite quotient of $G(X,r)$ in \cite{LV_StrGroups}, generalising the Coxeter-like quotients from \cite{DehCycleSet}; and to understand when $G(X,r)$ has torsion in \cite{JKVAV20}. Also, $\overline{J}$ can be used to transport the product of $A(X,r)$ to $M(X,r)$. Combining it with the original product of $M(X,r)$, one gets a ring-like structure called \emph{skew-brace} (see \cite{SkewBraces} and numerous recent references thereto). This allows for fruitful transport of techniques and intuitions from group and ring theories to the study of YBE solutions. 

In this paper, we show that $r_{\op}$ is actually a part of a whole family of solutions $r^{(k)}$ indexed by the reflections $k$ of $r$. We recover $r_{\op}$ as $r^{(\Id)}$, and we get $r$ itself by adding an element $*$ to $X$ and taking as $k$ the projection onto $*$ (Example \ref{EX:original}). In general, there are many more solutions (of variable complexity) in this family, which we illustrate with examples. All these solutions induce isomorphic actions of braid monoids/groups, and share many properties (invertibility, involutiveness, idempotence etc.). They have the same structure shelves. The structure monoids of these solutions are related by bijective maps
\[\overline{J^{(k)}} \colon M(X,r) \to M(X,r^{(k)})\]
generalising $\overline{J}$. These maps are not monoid $1$-cocycles in general, but satisfy a more involved property \eqref{E:JproductShort}. We expect that, for a wide class of solutions, similar bijections relate the structure groups of solutions from this family. We have a candidate for such bijections, but cannot prove that it does the job. 
 Finally, every monoid $M(X,r^{(k)})$ left-acts on $X$; this generalises the left action of the derived monoid, given on the generators by $a \cdot b = b \op a$.

In the second part of the paper, we take a closer look at the reflections for an involutive solution. We show that to check the reflection equation for an LND/RND solution, it is enough to compare the first/second coordinates of~\eqref{E:RE} only. As a consequence, a map commuting with all the $\lambda_a$'s is a reflection for an LND solution, and a map $k$ satisfying $\rho_{k(a)}=\rho_a$ for all $a \in X$ is a reflection for an RND solution. The LND statements appeared in \cite{SVW18}, while the RND ones are new. We also suggest two ways (a stronger and a weaker) of dividing reflections for an involutive solution into equivalence classes, making the study of the (generally huge) set of reflections more organised. We show that two reflections equivalent in the stronger sense yield isomorphic solutions $r^{(k)}$. On the other hand, non-equivalent reflections might give the same solutions $r^{(k)}$. 

Finally, given an LND involutive solution and a map $k \colon X \to X$, we construct a binary operation $b \ast a= \lambda_b k \lambda_b^{-1} (a)$. We show that $k$ is a reflection if and only if $\ast$ satisfies the \emph{strange} property
\[(a \ast b) \ast a = b \ast a\]
for all $a,b \in X$. Moreover, $k$ can be uniquely reconstructed from $\ast$.

The constructions, results and proofs in this paper are given using a mixture of two languages: algebraic formulas, and the graphical calculus developed in \cite{LebVen} (which we use freely with a minimum of explanations; the reader is referred to \cite{LebVen} for more detail). The formulas we get are often difficult to digest; their pictorial analogues are on the contrary intuitive, and explain their origin.

\section{Reflections as parameters for generalised derived solutions}

Let $(X,r)$ be a right-non-degenerate (RND) set-theoretic solution to the Yang--Baxter equation (further simply called \emph{solution}). Let $k$ be a reflection for $(X,r)$. We will use $k$ to construct another solution $r^{(k)}$ on $X$. For $k=\Id_X$ we recover the derived solution. We then describe bijective maps $J^{n;k} \colon X^{ n} \to X^{ n}$ intertwining the braid monoid/group actions on $X^{ n}$ coming from $r$ and $r^{(k)}$ respectively; the case $k=\Id_X$ yields the guitar map from \cite{Sol,LYZ,LebVen}. Thus with any RND solution comes a whole family of solutions parametrised by its reflections. These solutions vary in complexity, but all yield isomorphic braid monoid/group actions, the same structure shelf, and, as we shall see in the next section, intimately related structure monoids. 

Now let us turn to concrete statements.

\begin{defn}\label{D:k-sol}
The \emph{generalised derived solution} associated to a reflection $k$ for an RND solution $(X,r)$, or simply a $k$-\emph{derived solution}, is the map
\begin{align*}
r^{(k)} \colon X \times X &\longrightarrow X \times X,\\
(a,b) &\longmapsto (b'= \rho_{k(a')}\lambda_{\rho_{k(b)}^{-1}(a)}(b)\ ,\ a'=  \rho_b\rho_{k(b)}^{-1}(a)).
\end{align*}
\end{defn}

In Corollary~\ref{C:GenDerived} we will show that $r^{(k)}$ is indeed a solution.

\begin{exa}\label{EX:derived}
The trivial reflection $k=\Id_X$ yields the classical derived solution $(a,b) \longmapsto (\rho_a \lambda_{\rho_{b}^{-1}(a)}(b),a)$.
\end{exa}

\begin{exa}\label{EX:original}
Extend the solution $(X,r)$ to $X^*=X \sqcup \{*\}$ by imposing $\lambda_*=\rho_*=\Id_{X^*}$, and $\lambda_a(*)=\rho_a(*)=*$ for all $a \in X$. Consider the projection onto $\{*\}$: $k(a)=*$ for all $a \in X^*$. It is a reflection, since both sides of \eqref{E:RE} yield $(*,*)$ for all $a,b \in X^*$. This reflection yields $r^{(k)}(a,b)=(\lambda_a(b),\rho_b(a))$, which is simply the original solution extended to $X^*$.
\end{exa}

In the next example, on the contrary, the generalised derived solution is very different from the original one:
\begin{exa}\label{EX:3}
Consider the solution $X=\{1,2,3\}$, $\lambda_3=\rho_3=(12)$, and its reflection $k(a)=3$ for all $a$. Then for $r^{(k)}$ we have $\lambda_1=\rho_1=\lambda_2=\rho_2=(12)$. Here and below the omitted $\lambda$'s and $\rho$'s are all identities.
\end{exa}

In the last example, we see that even small solutions may have many non-isomorphic generalised derived solutions of very different nature.
\begin{exa}\label{EX:long}
Consider the solution $X=\{1,2,3,4\}$, $\lambda_a=(132)$ for all $a$, $\rho_1=(13), \rho_2=(12), \rho_3=(23), \rho_4=(123)$. To see that it is a solution, consider the map 
\[r(a,b)=(b-1,-a-b-1) \text{ on } \{1,2,3\},\]
where the set $\{1,2,3\}$ is identified with $\Z/3$. One checks that 
\[r_1r_2r_1(a,b,c)=(c+1,-b-c+1,a+b-c)=r_2r_1r_2(a,b,c).\]
In fact it is sufficient to compute the first two coordinates only (which are the simplest ones), since $r$ preserves the signed sum of its entries:
\[(b-1)-(-a-b-1)=a-b.\]
Further, the map $\sigma=(132)$ is compatible with this map $r$, in the sense of $r(\sigma \times \sigma)=(\sigma \times \sigma)r$. Thus one can extend $r$ to $X=\{1,2,3,4\}$ by putting $\lambda_4=\sigma$ and $\rho_4=\sigma^{-1}$. This is precisely the solution described above. 

\noindent One checks that the maps of the following type are reflections for our solution:
\begin{itemize}
\item projections $p=aaaa$ for all $a$;
\item maps $p=aaa4$ for all $a$;
\item maps $p=444a$ for all $a$.
\end{itemize} 
Here and below when writing $k=abcd$ we mean $k(1)=a, k(2)=b$, $k(3)=c$, $k(4)=d$. Together with $k=\Id$, this sums up to eleven reflections. A computer-aided verification shows that this is the complete list.

\noindent The reflection $k=\Id$ produces the derived solution given by  
	\begin{align*}
	\rho_a=\Id \text{ for all } a, && \lambda_a=(a-1,a+1) \text{ for all } a \neq 4, \ \lambda_4=\Id.
	\end{align*}
Here and below $a \pm 1$ is computed modulo $3$.	

\noindent The projection $k=4444$ produces the solution given by
\begin{align*}
	\rho_a=(a-1,a+1) \text{ for all } a \neq 4, \ \rho_4=\Id,&&\lambda_a=\Id \text{ for all } a.
\end{align*}
This is the mirror variant of the preceding solution.

\noindent The projection $k=1111$ produces the solution given by
\begin{align*}
	\rho_1=\id,\ \rho_2=(132),\ \rho_3=(123),\ \rho_4=(23),&&\lambda_a=(23) \text{ for all } a.
\end{align*}

\noindent The reflection $k=1114$ produces the solution given by 
\begin{align*}
	\rho_1=\rho_4=\Id, && \rho_2=(132), && \rho_3=(123), && \lambda_a=(23) \text{ for all } a \neq 4, \lambda_4=\Id.
\end{align*}
It differs from the preceding solution by the values of $\rho_4$ and $\lambda_4$.

\noindent These four solutions are clearly non-isomorphic. In fact, according to our computer, this is a complete list of generalised derived solutions in this case, up to isomorphism.			
\end{exa}

Already our central Definition~\ref{D:k-sol} is difficult to grasp in the algebraic form. Let us give its pictorial interpretation.

Graphically, the solution $r$ and the reflection $k$ are represented by a crossing and a bead respectively. The YBE then becomes the Reidemeister $\mathrm{III}$ move (Fig.~\ref{P:YBE}), in the sense that, putting the same colors at the bottom of the two diagrams, one propagates them upwards to get the same colors at the top. In other words, the Reidemeister $\mathrm{III}$ move has only a local effect on colorings. Here and below we mark in green the ``starting'' colors of a diagram, which uniquely determine the colors of all the remaining diagram parts. The reflection equation is depicted (in the same sense as the YBE) on Fig.~\ref{P:RE}. It does not seem natural drawn this way, but it becomes so if one thinks of the beads as the twists around a straight pole, as shown in the right part of Fig.~\ref{P:RE}. In what follows we will use the beads rather than the pole in order to make our diagrams lighter.

\begin{figure}[h]
\centering
\begin{tikzpicture}[scale=0.7]
\draw [rounded corners](0,0)--(0,0.25)--(0.4,0.4);
\draw [rounded corners](0.6,0.6)--(1,0.75)--(1,1);
\draw [rounded corners](1,0)--(1,0.25)--(0,0.75)--(0,1);
\node  at (0,-0.4) [mygreen]  {$a$};
\node  at (1,-0.4) [mygreen] {$b$};
\node  at (-.2,1) [above] {$\lambda_a(b)$};
\node  at (1.2,1) [above] {$\rho_b(a)$};
\draw [|->, red, thick]  (-1.5,0) -- (-1.5,1);
\node  at (2,.5){ };
\end{tikzpicture}
\begin{tikzpicture}[scale=0.7]
\draw (0,0)--(0,1);
\node  at (0,-0.4) [mygreen] {$a$};
\node  at (0,1) [above] {$k(a)$};
\fill (0,.5) circle (.1);
\node  at (3,.5){ };
\end{tikzpicture}
\begin{tikzpicture}[xscale=0.45,yscale=0.45]
\draw [rounded corners](0,0)--(0,0.25)--(0.4,0.4);
\draw [rounded corners](0.6,0.6)--(1,0.75)--(1,1.25)--(1.4,1.4);
\draw [rounded corners](1.6,1.6)--(2,1.75)--(2,3);
\draw [rounded corners](1,0)--(1,0.25)--(0,0.75)--(0,2.25)--(0.4,2.4);
\draw [rounded corners](0.6,2.6)--(1,2.75)--(1,3);
\draw [rounded corners](2,0)--(2,1.25)--(1,1.75)--(1,2.25)--(0,2.75)--(0,3);
\node  at (4,1.5){\Large $\overset{YBE}{=}$};
\node  at (4,-.1){ };
\end{tikzpicture}
\begin{tikzpicture}[xscale=0.45,yscale=0.45]
\draw [rounded corners](1,1)--(1,1.25)--(1.4,1.4);
\draw [rounded corners](1.6,1.6)--(2,1.75)--(2,3.25)--(1,3.75)--(1,4);
\draw [rounded corners](0,1)--(0,2.25)--(0.4,2.4);
\draw [rounded corners](0.6,2.6)--(1,2.75)--(1,3.25)--(1.4,3.4);
\draw [rounded corners](1.6,3.6)--(2,3.75)--(2,4);
\draw [rounded corners](2,1)--(2,1.25)--(1,1.75)--(1,2.25)--(0,2.75)--(0,4);
\node  at (2.3,.9){ };
\end{tikzpicture}
\caption{Graphical depiction of a solution, its reflexion, and the Yang--Baxter equation}\label{P:YBE}
\end{figure}
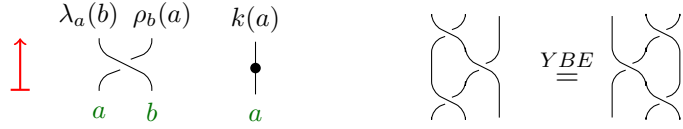

\begin{figure}[h]
\centering
\begin{tikzpicture}[scale=0.6]
\draw [rounded corners](0,-.4)--(0,0.25)--(0.4,0.4);
\draw [rounded corners](0.6,0.6)--(1,0.75)--(1,1.25)--(0,1.75)--(0,2.4);
\draw [rounded corners](1,-.4)--(1,0.25)--(0,0.75)--(0,1.25)--(0.4,1.4);
\draw [rounded corners] (0.6,1.6)--(1,1.75)--(1,2.4);
\fill (1,1) circle (.15);
\fill (1,0) circle (.15);
\node  at (2.3,1){\Large $\overset{RE}{=}$};
\node  at (2.5,-1.5){};
\end{tikzpicture}
\begin{tikzpicture}[scale=0.6]
\draw [rounded corners](0,-.4)--(0,0.25)--(0.4,0.4);
\draw [rounded corners](0.6,0.6)--(1,0.75)--(1,1.25)--(0,1.75)--(0,2.4);
\draw [rounded corners](1,-.4)--(1,0.25)--(0,0.75)--(0,1.25)--(0.4,1.4);
\draw [rounded corners] (0.6,1.6)--(1,1.75)--(1,2.4);
\fill (1,1) circle (.15);
\fill (1,2) circle (.15);
\node  at (5,-1.5){};
\end{tikzpicture}
\begin{tikzpicture}[xscale=0.55,yscale=0.5]
\draw [rounded corners](0,-2)--(0,0.25)--(0.4,0.4);
\draw [rounded corners](0.6,0.6)--(1,0.75)--(1,1.25)--(1.4,1.4);
\draw [rounded corners](1.6,1.6)--(2,1.75)--(2,2.25)--(1,2.75)--(1,3.25)--(0,3.75)--(0,4);
\draw [rounded corners](1,-2)--(1,-1.75)--(1.4,-1.6);
\draw [rounded corners](1.6,-1.4)--(2,-1.25)--(2,-.75)--(1,-.25)--(1,0.25)--(0,0.75)--(0,3.25)--(0.4,3.4);
\draw [rounded corners] (0.6,3.6)--(1,3.75)--(1,4);
\draw [color=myviolet,line width=.5mm] (1.5,-2)--(1.5,-.6);
\draw [color=myviolet,line width=.5mm] (1.5,-.4)--(1.5,2.4);
\draw [color=myviolet,line width=.5mm] (1.5,2.6)--(1.5,4);
\node  at (3.5,1){\Large $\overset{RE}{=}$};
\node  at (5,0){};
\end{tikzpicture}
\begin{tikzpicture}[xscale=0.55,yscale=0.5]
\draw [rounded corners](0,-2)--(0,-1.75)--(0.4,-1.6);
\draw [rounded corners](0.6,-1.4)--(1,-1.25)--(1,-.75)--(1.4,-.6);
\draw [rounded corners](1.6,-.4)--(2,-.25)--(2,.25)--(1,0.75)--(1,1.25)--(0,1.75)--(0,4);
\draw [rounded corners](1,-2)--(1,-1.75)--(0,-1.25)--(0,1.25)--(0.4,1.4);
\draw [rounded corners] (0.6,1.6)--(1,1.75)--(1,2.25)--(1.4,2.4);
\draw [rounded corners] (1.6,2.6)--(2,2.75)--(2,3.25)--(1,3.75)--(1,4);
\draw [color=myviolet,line width=.5mm] (1.5,-2)--(1.5,.4);
\draw [color=myviolet,line width=.5mm] (1.5,.6)--(1.5,3.4);
\draw [color=myviolet,line width=.5mm] (1.5,3.6)--(1.5,4);
\end{tikzpicture}
\caption{A short-hand and a detailed graphical depictions of the reflection equation}\label{P:RE}
\end{figure}
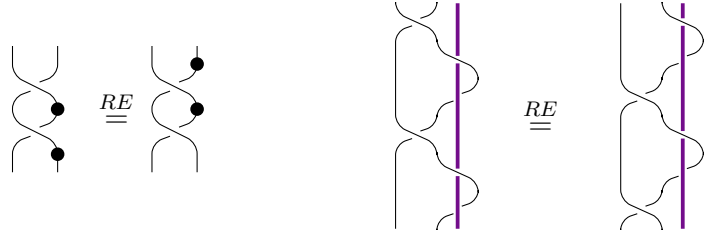

Now, in the left diagram of Fig.~\ref{P:GDsolution}, the (green) colors $a$ and $b$ put just below the two beads uniquely determine all the remaining colors, in particular the bottom left color $c=\rho_{k(b)}^{-1}(a)$. Then in the right diagram, the colors $c$ and $b$ put at the bottom uniquely determine all the remaining colors, in particular the colors $a'=\rho_b(c)$ and $b'=\rho_{k(a')}\lambda_{c}(b)$ just below the beads. The map $r^{(k)}$ from Definition~\ref{D:k-sol} works with the colors just below the beads: it sends $(a,b)$ to $(b',a')$. For future reference, note that the top colors of the two diagrams of Fig.~\ref{P:GDsolution} coincide due to the reflection equation.

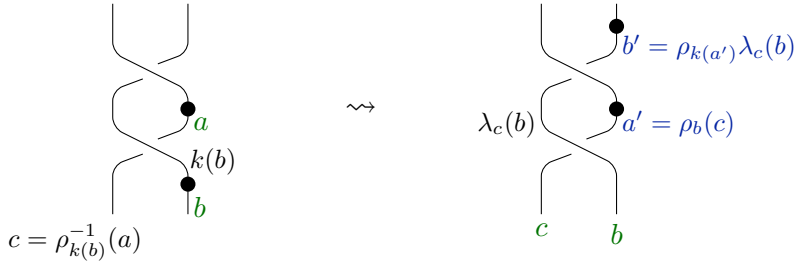
\begin{figure}[h]
\centering
\begin{tikzpicture}[xscale=1,yscale=1]
\draw [rounded corners](0,-.4)--(0,0.25)--(0.4,0.4);
\draw [rounded corners](0.6,0.6)--(1,0.75)--(1,1.25)--(0,1.75)--(0,2.4);
\draw [rounded corners](1,-.4)--(1,0.25)--(0,0.75)--(0,1.25)--(0.4,1.4);
\draw [rounded corners] (0.6,1.6)--(1,1.75)--(1,2.4);
\fill (1,1) circle (.1);
\fill (1,0) circle (.1);
\node  [mygreen,right] at (.95,-0.3){\large $b$};
\node  [mygreen,right] at (.95,.8){\large $a$};
\node  [right] at (.9,.3){$k(b)$};
\node  [below] at (-.5,-.4){$c=\rho_{k(b)}^{-1}(a)$};
\node  at (3.3,1){\Large $\leadsto$};
\node  at (4.5,-1.5){};
\end{tikzpicture}
\begin{tikzpicture}[xscale=1,yscale=1]
\draw [rounded corners](0,-.4)--(0,0.25)--(0.4,0.4);
\draw [rounded corners](0.6,0.6)--(1,0.75)--(1,1.25)--(0,1.75)--(0,2.4);
\draw [rounded corners](1,-.4)--(1,0.25)--(0,0.75)--(0,1.25)--(0.4,1.4);
\draw [rounded corners] (0.6,1.6)--(1,1.75)--(1,2.4);
\fill (1,1) circle (.1);
\fill (1,2.1) circle (.1);
\node  [mygreen,below] at (0,-.4){\large $c$};
\node  [mygreen,below] at (1,-.4){\large $b$};
\node  [left] at (.05,.8){$\lambda_{c}(b)$};
\node  [myblue,right] at (.95,.8){$a'=\rho_b(c)$};
\node  [myblue,right] at (.95,1.8){$b'=\rho_{k(a')}\lambda_{c}(b)$};
\node  at (5,-1.5){};
\end{tikzpicture}
\caption{The generalised derived solution $(a,b) \mapsto (b',a')$.}\label{P:GDsolution}
\end{figure}

Now, we have to prove that our maps $r^{(k)}$ are indeed YBE solutions. This fact, as well as the fundamental properties of these solutions, are based on the following construction.

\begin{defn}
Given a solution $(X,r)$ and its reflection $k$, the \emph{$k$-Garside maps} are the following operators on $X^{n}$:
\begin{align*}
\Delta^{n;k} &= k_n\ (r_{n-1}\cdots r_2r_1)\ k_n \cdots k_n \ (r_{n-1}r_{n-2})\ k_n \ r_{n-1}\ k_n.
\end{align*}
 The \emph{$k$-guitar maps} are then defined by
\begin{align*}
J^{n;k} \colon X^{n} & \longrightarrow X^{n},\\
(a_1,a_2,\ldots,a_n) &\longmapsto (\rho_{\Delta^{n-1;k}(a_2,\ldots,a_n)}(a_1), \ \rho_{\Delta^{n-2;k}(a_3,\ldots,a_n)}(a_2),\ldots,\\ 
&\hspace*{2.5cm}\rho_{\Delta^{2;k}(a_{n-1},a_n)}(a_{n-2}),\ \rho_{k(a_n)}(a_{n-1}),\ a_n).
\end{align*}
Here we extended the maps $\rho$ to 
\begin{equation}\label{E:rho_ov_b}
\rho_{(b_1,b_2,\ldots,b_m)} =\rho_{b_m} \cdots \rho_{b_2}\rho_{b_1}.
\end{equation}
\end{defn}

The maps $\Delta^{n;k}$ resemble the action of the classical Garside elements of the positive braid monoids, except for the reflection $k$ repeatedly applied to rightmost elements, hence the name.

The maps $\Delta^{n;k}$ and $J^{n;k}$ are graphically interpreted in Fig.~\ref{P:Guitar}. As with the generalised derived solution $r^{(k)}$, the important colors to look at are the ones just below the beads.

\begin{figure}[h]
\centering
\begin{tikzpicture}[xscale=1.4,yscale=1.4]
 \draw [rounded corners=10] (0,0) -- (5,2.5) -- (4,3);
 \draw [line width=4pt,white, rounded corners=10] (1,0) -- (5,2) -- (3,3); 
 \draw [rounded corners=10] (1,0) -- (5,2) -- (3,3); 
 \draw [line width=4pt,white, rounded corners=10] (2,0) -- (5,1.5) -- (2,3); 
 \draw [rounded corners=10] (2,0) -- (5,1.5) -- (2,3);
 \draw [line width=4pt,white, rounded corners=10] (3,0) -- (5,1) -- (1,3); 
 \draw [rounded corners=10] (3,0) -- (5,1) -- (1,3); 
 \draw [line width=4pt,white, rounded corners=10] (4,0) -- (5,0.5) -- (0,3);   
 \draw [rounded corners=10] (4,0) -- (5,0.5) -- (0,3);  
 \node at (0,-0.4) [mygreen,above] {${a_1}$};  
 \node at (1,-0.4) [mygreen,above] {${a_2}$}; 
 \node at (2,-0.4) [mygreen,above] {${\cdots}$}; 
  \node at (3,-0.4) [mygreen,above] {${a_{n-1}}$}; 
 \node at (4,-0.4) [mygreen,above] {${a_n}$}; 
 \node at (4.7,2.3) [myblue,right] {$a'_1$};   
 \node at (4.7,1.8) [myblue,right] {$a'_2$}; 
 \node at (4.8,1.5) [myblue,right] {$\;\scriptstyle{\vdots}$};  
 \node at (4.7,.8) [myblue,right] {$a'_{n-1}$};   
 \node at (4.7,.3) [myblue,right] {$a'_n$};  
 \fill (4.85,2.5) circle (.08); 
 \fill (4.85,2) circle (.08);  
 \fill (4.85,1.5) circle (.08);  
 \fill (4.85,1) circle (.08);  
 \fill (4.85,.5) circle (.08);  
 \node at (7,1.5) [myblue] {$J^{n;k}(a_1,a_2,\ldots,a_{n-1},a_n)$}; 
 \node at (7,1) [myblue] {$=(a'_1,a'_2,\ldots,a'_{n-1},a'_n)$}; 
 \node at (2,3.5) [myred] {$\Delta^{n;k}(a_1,a_2,\ldots,a_{n-1},a_n)$};  
 \draw [line width=1.5pt,myred] (-.2,3) -- (-.2,3.2) -- (4.2,3.2) -- (4.2,3); 
 \draw [line width=1.5pt,myblue] (5.2,2.7) -- (5.4,2.7) -- (5.4,0) -- (5.2,0);  
\end{tikzpicture} 
   \caption{The $k$-Garside map $\Delta^{n;k}$ and the $k$-guitar map $J^{n;k}$.}\label{P:Guitar}
\end{figure}
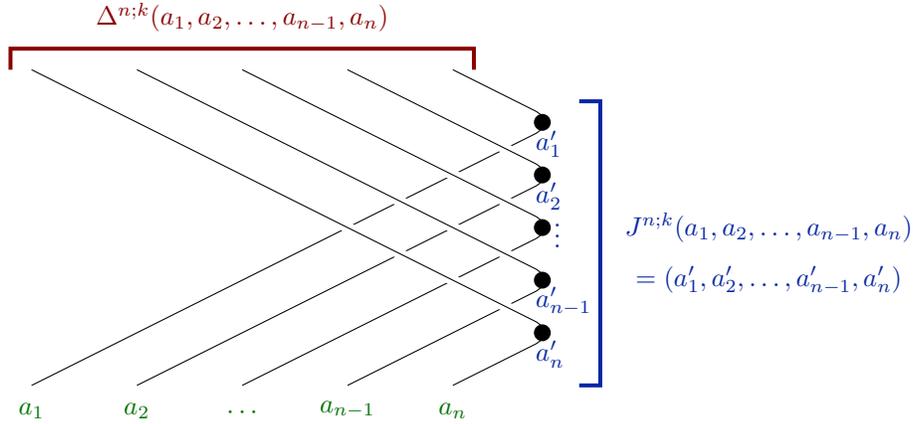

\begin{lem}
The maps $J^{n;k}$ are bijective if $r$ is RND.
\end{lem}

\begin{proof}
If all the $\rho_a$'s are bijective, then from 
\[J^{n;k}(a_1,a_2,\ldots,a_n)=(a'_1,a'_2,\ldots,a'_n)\] 
one can reconstruct $a_n=a'_n$, then $a_{n-1}=\rho_{k(a_n)}^{-1}(a'_{n-1})$, and so on until 
\[a_1=\rho_{\Delta^{n-1;k}(a_2,\ldots,a_n)}^{-1}(a'_1).\qedhere\]
\end{proof}

\begin{thm}\label{T:GenDerived}
Let $(X,r)$ be an RND solution, and $k$ its reflection. Then the $k$-guitar maps intertwine the solution~$r$ and the $k$-derived solution $r^{(k)}$. That is, for all $n \in \N$ and all $1 \leq i <n$, one has
\begin{align}
J^{n;k}r_i &= r^{(k)}_iJ^{n;k}.\label{E:JandR}
\end{align}
As for the $k$-Garside map, it is compatible with $r$ in the following sense:
\begin{align}
\Delta^{n;k}r_i &= r_{n-i}\Delta^{n;k}.\label{E:DeltaandR}
\end{align}
\end{thm}

This theorem directly yields the following fundamental properties of generalised derived solutions:
\begin{cor}\label{C:GenDerived}
Let $(X,r)$ be an RND solution, and $k$ its reflection. Then:
\begin{enumerate}
\item $r^{(k)}$ is an RND YBE solution;
\item $r$ and $r^{(k)}$ induce isomorphic braid monoid/group actions;
\item the relation $r^s=r^t$ for some $s>t\geq 0$ is equivalent to $(r^{(k)})^s=(r^{(k)})^t$; in particular, $r$ is invertible/idempotent if and only if $r^{(k)}$ is so; 
\item $r$ is invertible if and only if $r^{(k)}$ is so. 
\end{enumerate}
\end{cor}

\begin{proof}[Proof of Theorem~\ref{T:GenDerived}]
A graphical proof is given in Fig.~\ref{P:GuitarBr}. 

\begin{figure}[h]
\centering
\begin{tikzpicture}[xscale=0.75,yscale=1,>=latex]
 \draw [rounded corners=10] (0,-0.5) -- (0,.25) -- (5,2.75) -- (4.5,3);
 \draw [line width=4pt,white, rounded corners=10] (1,0) -- (5,2) -- (3,3); 
 \draw [rounded corners=10] (2,0.5) -- (5,2) -- (3,3);
 \draw [line width=4pt,white, rounded corners=10] (2,0) -- (5,1.5) -- (2,3); 
 \draw [rounded corners=10] (1,-0.5) -- (2,0) -- (5,1.5) -- (2,3); 
 \draw [line width=4pt,white, rounded corners=10] (4,0) -- (5,0.5) -- (0,3);   
 \draw [rounded corners=10] (4,-0.5) -- (4,0) -- (5,0.5) -- (0,3);  
 \draw [line width=4pt,white, rounded corners=10] (2,-0.5) -- (1,0) -- (2,0.5); 
 \draw [rounded corners=10] (2,-0.5) -- (1,0) -- (2,0.5); 
 \draw [myviolet, dashed] (-.8,0)--(5,0); 
 \node at (-.7,-0.5) [myviolet] {${r_2}$};  
 \node at (-.7,1.5) [myviolet]  {${J}$}; 
 \node at (0,-0.9) [above,mygreen] {$\scriptstyle{a_1}$};  
 \node at (1,-0.9) [above,mygreen] {$\scriptstyle{a_2}$}; 
 \node at (2,-0.9) [above,mygreen] {$\scriptstyle{a_3}$}; 
 \node at (4,-0.9) [above,mygreen] {$\scriptstyle{a_4}$};
 \node at (4.5,2.55) [right] {$\scriptstyle{\color{myblue}J_1(r_2(\overline{a}))}$};   
 \node at (4.5,1.8) [right] {$\scriptstyle{\color{myblue}J_2(r_2(\overline{a}))}$};   
 \node at (4.5,1.3) [right] {$\scriptstyle{\color{myblue}J_3(r_2(\overline{a}))}$};  
 \node at (4.5,0.3) [right] {$\scriptstyle{\color{myblue}J_4(r_2(\overline{a}))}$};  
 \fill (4.75,2.75) circle (.08); 
 \fill (4.75,2) circle (.08);  
 \fill (4.75,1.5) circle (.08);  
 \fill (4.75,.5) circle (.08); 
 \draw [latex->, rounded corners=10,myviolet] (2,-1) -- (2,-1.5)node[right] {YBE} -- (2,-2);
\end{tikzpicture} 
\hspace*{20pt}
\begin{tikzpicture}[xscale=0.75,yscale=1,>=latex]
 \draw [rounded corners=10] (0,-0.5) -- (0,0) -- (5,2.5) -- (4,3);
 \draw [line width=4pt,white, rounded corners=10] (2,-0.5) -- (5,1) -- (2,2.5) -- (3,3);
 \draw [rounded corners=10] (2,-0.5) -- (5,1) -- (2,2.5) -- (3,3);
 \draw [line width=4pt,white, rounded corners=10] (2,0) -- (5,1.5) -- (2,3); 
 \draw [rounded corners=10] (1,-0.5) -- (2,0) -- (5,1.5) -- (2,3); 
 \draw [line width=4pt,white, rounded corners=10] (4,-0.5) -- (5,0) -- (0,2.5) -- (0,3);   
 \draw [rounded corners=10] (4,-0.5) -- (5,0) -- (0,2.5) -- (0,3);    
 \draw [line width=4pt,white, rounded corners=10] (4.75,1.125) -- (4.25,1.375);   
 \draw [rounded corners=10] (4.75,1.125) -- (4.25,1.375); 
 \draw [myviolet, dashed] (-0.5,2.65)--(5.5,2.65); 
 \node at (.5,1.5) [myviolet]  {${J}$}; 
 \node at (0,-0.9) [above,mygreen] {$\scriptstyle{a_1}$};  
 \node at (1,-0.9) [above,mygreen] {$\scriptstyle{a_2}$}; 
 \node at (2,-0.9) [above,mygreen] {$\scriptstyle{a_3}$}; 
 \node at (4,-0.9) [above,mygreen] {$\scriptstyle{a_4}$};
 \node at (4.5,2.25) [right] {$\scriptstyle{\color{red} J_1(\overline{a})}$};   
 \node at (4.5,1.3) [right] {$\scriptstyle{\color{red} J_2(\overline{a})}$};  
 \node at (4.5,.8) [right] {$\scriptstyle{\color{red} J_3(\overline{a})}$};   
 \node at (4.5,-.2) [right] {$\scriptstyle{\color{red} J_4(\overline{a})}$};   
 \fill (4.75,2.5) circle (.08); 
 \fill (4.75,1) circle (.08);  
 \fill (4.75,1.5) circle (.08);  
 \fill (4.75,0) circle (.08);  
 \draw [latex->, rounded corners=10,myviolet] (2,-1) -- (2,-1.5)node[right] {YBE} -- (2,-2);
\end{tikzpicture}
\begin{tikzpicture}[xscale=0.75,yscale=1,>=latex]
 \draw [rounded corners=10] (0,-0.5) -- (0,.25) -- (5,2.75) -- (4.5,3);
 \draw [line width=4pt,white, rounded corners=10] (1,0) -- (5,2) -- (3,3); 
 \draw [rounded corners=10] (2,-0.5) -- (5,1) -- (4,1.5) -- (5,2) -- (3,3);
 \draw [line width=4pt,white, rounded corners=10] (2,0) -- (5,1.5) -- (2,3); 
 \draw [rounded corners=10] (1,-0.5) -- (2,0) -- (5,1.5) -- (2,3); 
 \draw [line width=4pt,white, rounded corners=10] (4,-0.5) -- (5,0) -- (0,2.5) -- (0,3);   
 \draw [rounded corners=10] (4,-0.5) -- (5,0) -- (0,2.5) -- (0,3);  
 \draw [line width=4pt,white]  (4.625,1.1875) -- (4.375,1.3125);    
 \draw (4.625,1.1875) -- (4.375,1.3125); 
 \node at (0,-0.9) [above,mygreen] {$\scriptstyle{a_1}$};  
 \node at (1,-0.9) [above,mygreen] {$\scriptstyle{a_2}$}; 
 \node at (2,-0.9) [above,mygreen] {$\scriptstyle{a_3}$}; 
 \node at (4,-0.9) [above,mygreen] {$\scriptstyle{a_4}$};
 \node at (4.5,2.55) [right] {$\scriptstyle{{\color{myblue}J_1(r_2(\overline{a}))}={\color{red} J_1(\overline{a})}}$};   
 \node at (4.5,1.8) [right] {$\scriptstyle{{\color{myblue}J_2(r_2(\overline{a}))}={\color{red} J_3(\overline{a})'}}$};   
 \node at (4.5,1.3) [right] {$\scriptstyle{{\color{myblue}J_3(r_2(\overline{a}))}={\color{red} J_2(\overline{a})'}}$}; 
 \node at (4.5,-.2) [right] {$\scriptstyle{{\color{myblue}J_4(r_2(\overline{a}))}={\color{red} J_4(\overline{a})}}$};  
 \fill (4.75,2.75) circle (.08); 
 \fill (4.75,2) circle (.08);  
 \fill (4.75,1.5) circle (.08);  
 \fill (4.75,0) circle (.08);   
 \path [fill=lightgrey, fill opacity=0.2] (4.75,.5)--(3,1.375)--(4.75,2.25)--(7.6,2.25)--(7.6,.5)--(4.75,.5);
\end{tikzpicture}
\hspace*{3pt}
\begin{tikzpicture}[xscale=0.75,yscale=1,>=latex]
 \draw [rounded corners=10] (0,-0.5) -- (0,.25) -- (5,2.75) -- (4.5,3); \draw [line width=4pt,white, rounded corners=10] (1,0) -- (5,2) -- (3,3); 
 \draw [rounded corners=10] (2,-0.5) -- (5,1) -- (4,1.5) -- (5,2) -- (3,3);
 \draw [line width=4pt,white, rounded corners=10] (2,0) -- (5,1.5) -- (2,3); 
 \draw [rounded corners=10] (1,-0.5) -- (2,0) -- (5,1.5) -- (2,3); 
 \draw [line width=4pt,white, rounded corners=10] (4,-0.5) -- (5,0) -- (0,2.5) -- (0,3);   
 \draw [rounded corners=10] (4,-0.5) -- (5,0) -- (0,2.5) -- (0,3);  
 \draw [line width=4pt,white]  (4.625,1.1875) -- (4.375,1.3125);    
 \draw (4.625,1.1875) -- (4.375,1.3125); 
 \node at (0,-0.9) [above,mygreen] {$\scriptstyle{a_1}$};  
 \node at (1,-0.9) [above,mygreen] {$\scriptstyle{a_2}$}; 
 \node at (2,-0.9) [above,mygreen] {$\scriptstyle{a_3}$}; 
 \node at (4,-0.9) [above,mygreen] {$\scriptstyle{a_4}$};
 \node at (4.5,2.55) [right] {$\scriptstyle{\color{red} J_1(\overline{a})}$};     
 \node at (4.5,1.3) [right] {$\scriptstyle{\color{red} J_2(\overline{a})}$};  
 \node at (4.5,.8) [right] {$\scriptstyle{\color{red} J_3(\overline{a})}$};   
 \node at (4.5,-.2) [right] {$\scriptstyle{\color{red} J_4(\overline{a})}$};  
 \fill (4.75,2.75) circle (.08); 
 \fill (4.75,1.5) circle (.08); 
 \fill (4.75,1) circle (.08);    
 \fill (4.75,0) circle (.08);   
 \draw [latex->, rounded corners=10,myviolet] (-1.5,1.5) -- (-.5,1.5)node[above] {$\scriptstyle{\text{definition of }r^{(k)}}$}node[below] {$\scriptstyle{\text{and RE}}$} -- (.5,1.5);
 \path [fill=lightgrey, fill opacity=0.2] (4.75,.5)--(3,1.375)--(4.75,2.25)--(5.6,2.25)--(5.6,.5)--(4.75,.5);
\end{tikzpicture}
   \caption{The entwining relation $J^{n;k}r_i=r^{(k)}_iJ^{n;k}$ (here $n=4$, $i=2$) is established by comparing the colors just below the beads in the bottom left diagram as calculated from the upper left (blue labels) and the two right diagrams (red labels). The relation $\Delta^{n;k}r_i = r_{n-i}\Delta^{n;k}$ is established by comparing the top colors of the two upper diagrams.}\label{P:GuitarBr}
\end{figure}
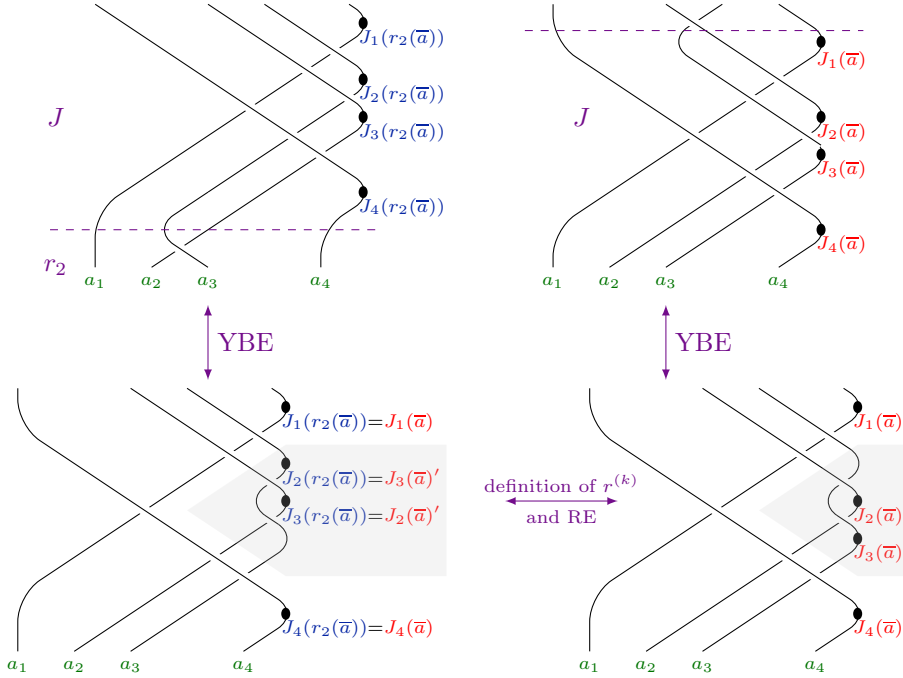

We start from the upper left diagram representing $J^{n;k}r_i$. The bottom colors $\overline{a}$ are fixed, and uniquely determine all the remaining colors.  We then pull the crossing corresponding to $r_i$ to the right, and get the lower left diagram. Due to the YBE, this does not alter the (blue) colors just below the beads. Then we make one bead slide around its neighbouring bead in the grey zone, and obtain the lower right diagram. This preserves the colors outside the grey zone because of the RE. And for the colors just below the beads, this operation boils down to applying $r^{(k)}$ (directed from right to left). Now, one can pull the beadless crossing from the grey zone upward, and get the upper right diagram. Once again, the YBE guarantees that the the (red) colors just below the beads remain unchanged. But from the upper right diagram, one sees that these colors are precisely $J^{n;k}(\overline{a})$.

The argument above shows that the colors at the top of the two upper diagrams coincide. But these are precisely the left and the right sides of
\eqref{E:DeltaandR}.
\end{proof}

From the proof it is clear that \eqref{E:DeltaandR} remains true for not necessarily RND solutions.

We have thus learnt that any RND solution comes with a whole family of closely related solutions parametrised by its reflections. It is natural to ask what happens if we iterate this construction, and look at the generalised derived solutions of generalised derived solutions. Computer experiments show that new iterations may produce new solutions, but the situation tends to stabilise fast. One of the reasons is the following result:

\begin{thm}\label{T:StrShelf}
Let $(X,r)$ be an RND solution, and $k$ its reflection. Then the structure shelf operations $\op$ and $\op_k$ for $r$ and $r^{(k)}$ respectively coincide. If moreover $k$ and $r$ are both invertible, then one has 
\[k(b \lop_k a)=k(a) \op k(b)\] 
for all $a,b \in X$. Here $\lop_k$ is the left structure shelf operations for $r^{(k)}$.
\end{thm}

Recall that the left structure shelf operations for an LND solution is given by the formula
\[b \lop a = \lambda_b \rho_{\lambda_a^{-1}(b)}(a).\]
In particular, under the assumptions of the proposition, the solution $r^{(k)}$ is automatically LND. Note that, according to~\cite{LV_StrGroups}, the right and left structure shelves of an invertible non-degenerate solution are isomorphic.

\begin{proof} 
In Fig.~\ref{P:GDsolutionStrShelf} we showed a double application of $r^{(k)}$. 

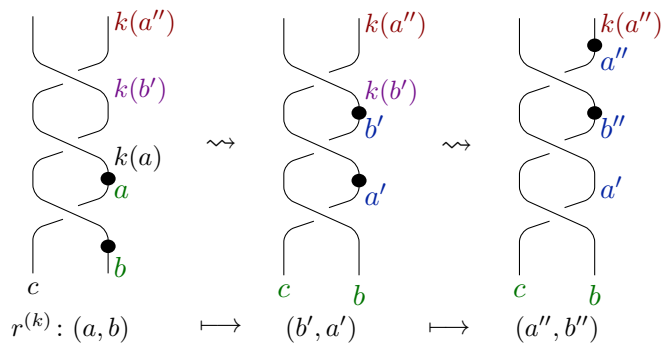
\begin{figure}[h]
\centering
\begin{tikzpicture}[xscale=1,yscale=.9]
\draw [rounded corners](0,-.4)--(0,0.25)--(0.4,0.4);
\draw [rounded corners](0.6,0.6)--(1,0.75)--(1,1.25)--(0,1.75)--(0,2.25)--(0.4,2.4);
\draw [rounded corners](0.6,2.6)--(1,2.75)--(1,3.4);
\draw [rounded corners](1,-.4)--(1,0.25)--(0,0.75)--(0,1.25)--(0.4,1.4);
\draw [rounded corners] (0.6,1.6)--(1,1.75)--(1,2.25)--(0,2.75)--(0,3.4);;
\fill (1,1) circle (.1);
\fill (1,0) circle (.1);
\node  [mygreen,right] at (.95,-0.3){\large $b$};
\node  [mygreen,right] at (.95,.8){\large $a$};
\node  [below] at (0,-.4){$c$};
\node  [myred,right] at (.95,3.3){ $k(a'')$};
\node  [myviolet,right] at (.95,2.3){$k(b')$};
\node  [right] at (.95,1.3){$k(a)$};
\node  at (2.5,1.5){\Large $\leadsto$};
\node  at (2.9,.5){};
\node  at (.5,-1.2){$r^{(k)} \colon (a,b)$};
\node  at (2.5,-1.2){$\longmapsto$};
\end{tikzpicture}
\begin{tikzpicture}[xscale=1,yscale=.9]
\draw [rounded corners](0,-.4)--(0,0.25)--(0.4,0.4);
\draw [rounded corners](0.6,0.6)--(1,0.75)--(1,1.25)--(0,1.75)--(0,2.25)--(0.4,2.4);
\draw [rounded corners](0.6,2.6)--(1,2.75)--(1,3.4);
\draw [rounded corners](1,-.4)--(1,0.25)--(0,0.75)--(0,1.25)--(0.4,1.4);
\draw [rounded corners] (0.6,1.6)--(1,1.75)--(1,2.25)--(0,2.75)--(0,3.4);;
\fill (1,1) circle (.1);
\fill (1,2) circle (.1);
\node  [myblue,right] at (.95,1.8){\large $b'$};
\node  [myviolet,right] at (.95,2.3){$k(b')$};
\node  [myblue,right] at (.95,.8){\large $a'$};
\node  [mygreen,below] at (0,-.4){\large $c$};
\node  [mygreen,below] at (1,-.4){\large $b$};
\node  [myred,right] at (.95,3.3){ $k(a'')$};
\node  at (2.3,1.5){\Large $\leadsto$};
\node  at (2.7,.5){};
\node  at (.5,-1.2){$(b',a')$};
\node  at (2.2,-1.2){$\longmapsto$};
\end{tikzpicture}
\begin{tikzpicture}[xscale=1,yscale=.9]
\draw [rounded corners](0,-.4)--(0,0.25)--(0.4,0.4);
\draw [rounded corners](0.6,0.6)--(1,0.75)--(1,1.25)--(0,1.75)--(0,2.25)--(0.4,2.4);
\draw [rounded corners](0.6,2.6)--(1,2.75)--(1,3.4);
\draw [rounded corners](1,-.4)--(1,0.25)--(0,0.75)--(0,1.25)--(0.4,1.4);
\draw [rounded corners] (0.6,1.6)--(1,1.75)--(1,2.25)--(0,2.75)--(0,3.4);;
\fill (1,3) circle (.1);
\fill (1,2) circle (.1);
\node  [myblue,right] at (.95,.9){\large $a'$};
\node  [myblue,right] at (.95,1.8){\large $b''$};
\node  [myblue,right] at (.95,2.8){\large $a''$};
\node  [myred,right] at (.95,3.3){ $k(a'')$};
\node  [mygreen,below] at (0,-.4){\large $c$};
\node  [mygreen,below] at (1,-.4){\large $b$};
\node  at (.5,-1.2){$(a'',b'')$};
\end{tikzpicture}
\caption{Comparing structure shelves for $r$ and $r^{(k)}$: the diagrams yield $b \op_k a' = b''= b \op a'$, $b' \lop_k a =a''$ and $k(a) \op k(b') = k(a'')$.}\label{P:GDsolutionStrShelf}
\end{figure}

As usual, the starting colors in each diagram are in green: $a$ and $b$ for the left one, and $c$ and $b$ for the remaining ones. The definition of $r^{(k)}$ yields 
\[r^{(k)} \colon (a,b) \mapsto (b',a') \mapsto (a'',b'').\] By the definition of structure shelves, this implies 
\[b \op_k a' =b'' \qquad \text{and} \qquad b' \lop_k a =a''.\] 
On the other hand, by the RE, the coloring changes due to a sliding bead are local. Hence the upper right colors are the same on the three diagrams. From the right diagram we see that it is $k(a'')$. In the same way, the color $k(b')$ can be transported from the middle to the left diagram, and the color $a'$ from the middle to the right diagram. Finally, from the right diagram and the definition of the right structure shelf we conclude 
\[b \op a' = b''= b \op_k a'.\] 
It remains to prove that this identity holds for all $b,a' \in X$. For this one needs to deduce the starting color $a$ from $a'$ and $b$. But this is easy:
\[a=\rho_{k(b)}(c)=\rho_{k(b)}\rho_b^{-1}(a').\]

Similarly, from the left diagram follows
\[k(a) \op k(b') = k(a'') =k(b' \lop_k a).\]
The starting color $b$ can be deduced from $a$ and $b'$ only when $k$ and $r$ are invertible. Indeed, the colors $k(a)$ and $k(b')$ on the left diagram can be propagated everywhere, and uniquely determine the remaining colors, in particular the bottom right color $b$. 
\end{proof}

\section{Structure monoids for generalised derived solutions}

The classical guitar maps yield a bijective monoid $1$-cocycle 
\[M(X,r) \to A(X,r)=M(X,r^{(\Id)}).\] 
In this section, we will see that the $k$-guitar maps induce a bijection $M(X,r) \to M(X,r^{(k)})$, 
satisfying a more general cocycle-like property \eqref{E:JproductShort}. Moreover, each structure monoid $M(X,r^{(k)})$ acts on the set $X$.

Recall the definition \eqref{E:StrMonoid} of the structure monoid of $(X,r)$. It can be reinterpreted as the quotient of the free monoid on $X$ by the relations $r_i(a_1,\ldots,a_n)=(a_1,\ldots,a_n)$ for all $n$, all $i<n$, and all $a_j \in X$. Thus the entwining relations $J^{n;k}r_i=r^{(k)}_iJ^{n;k}$ and $\Delta^{n;k}r_i = r_{n-i}\Delta^{n;k}$ from Theorem~\ref{T:GenDerived}, and the invertibility of $J^{n;k}$, imply

\begin{pro}\label{P:GuitarForMonoids}
Let $(X,r)$ be an RND solution, and $k$ its reflection. The $k$-guitar maps $J^{n;k}$ induce a bijection of $\Z_{\geq 0}$-graded sets 
\[\overline{J^{(k)}} \colon M(X,r) \to M(X,r^{(k)}).\]
Further, the $k$-Garside maps $\Delta^{n;k}$ induce a map of $\Z_{\geq 0}$-graded sets 
\[\overline{\Delta^{(k)}} \colon M(X,r) \to M(X,r).\]
\end{pro}

The map $\overline{J^{(k)}}$ is not necessarily a monoid $1$-cocycle. It interacts with the products in $M(X,r)$ and in $M(X,r^{(k)})$ in a more intricate way:
 
\begin{pro}\label{P:Almost1cocycle}
Let $(X,r)$ be a solution, and $k$ its reflection. Take $p,q \in \N$, $\overline{a} \in X^p$, $\overline{b} \in X^q$. One has:  
\begin{align}
J^{p+q;k}(\overline{a}\overline{b}) &= J^{p;k}(\rho_{\Delta^{q;k}(\overline{b})}(\overline{a}))\ J^{q;k}(\overline{b}),\label{E:Jproduct}\\
\Delta^{p+q;k}(\overline{a}\overline{b}) &= \lambda_{\overline{a}}(\Delta^{q;k}(\overline{b}))\ \Delta^{p;k}(\rho_{\Delta^{q;k}(\overline{b})}(\overline{a})).\label{E:Dproduct}
\end{align}
\end{pro}

Without the superscripts, the formulas take a more readable form:
\begin{align}
J(\overline{a}\overline{b}) &= J(\rho_{\Delta(\overline{b})}(\overline{a}))\ J(\overline{b}),\label{E:JproductShort}\\
\Delta(\overline{a}\overline{b}) &= \lambda_{\overline{a}}(\Delta(\overline{b}))\ \Delta(\rho_{\Delta(\overline{b})}(\overline{a})).\label{E:DproductShort}
\end{align}

Here $\lambda$ and $\rho$ are extended to the powers of~$X$ as shown in Fig~\ref{P:TX}. Note that the definition~\eqref{E:rho_ov_b} is a particular case of this extension.

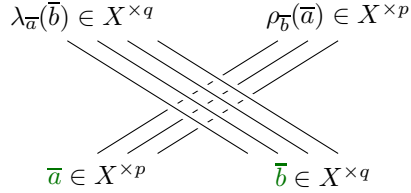
\begin{figure}[h]
\begin{tikzpicture}[xscale=0.4,yscale=0.25]
\draw (1,0)--(7,6);
\draw (2,0)--(8,6);
\draw (3,0)--(9,6);
\draw [line width=4pt,white] (6,0)--(0,6);
\draw [line width=4pt,white] (7,0)--(1,6);
\draw [line width=4pt,white] (8,0)--(2,6);
\draw [line width=4pt,white] (9,0)--(3,6);
\draw (6,0)--(0,6);
\draw (7,0)--(1,6);
\draw (8,0)--(2,6);
\draw (9,0)--(3,6);
\node [below] at (1,0) {${\color{mygreen}\overline{a}} \in X^{\times p}$};
\node [below] at (8.5,0) {${\color{mygreen}\overline{b}} \in X^{\times q}$};
\node [above] at (9,6) {$\rho_{\overline{b}}(\overline{a}) \in X^{\times p}$};
\node [above] at (0.5,6) {$\lambda_{\overline{a}}(\overline{b}) \in X^{\times q}$};
\end{tikzpicture}
\caption{The solution~$r$ extended to~$T(X)$.}\label{P:TX}
\end{figure}

\begin{rem} 
The YBE for $r$ implies that these extended $\lambda$ and $\rho$ induce left and right actions of $M(X,r)$ on itself respectively. Thus formulas \eqref{E:JproductShort}--\eqref{E:DproductShort} are valid for all $\overline{a},\overline{b} \in M(X,r)$, and the induced maps $\overline{J}$ and $\overline{\Delta}$.
\end{rem}

\begin{proof}[Proof of Proposition \ref{P:Almost1cocycle}] The proof is given on Fig.~\ref{P:GuitarProduct}. The dotted lines mean that we are simply reading the colors of the strands we intersect, without modifying them.

\begin{figure}[h]
\centering
\begin{tikzpicture}[xscale=1.4,yscale=1.4]
 \draw [rounded corners=10] (0,0) -- (5,2.5) -- (4,3);
 \draw [line width=4pt,white, rounded corners=10] (1,0) -- (5,2) -- (3,3); 
 \draw [rounded corners=10] (1,0) -- (5,2) -- (3,3); 
 \draw [line width=4pt,white, rounded corners=10] (3,0) -- (5,1) -- (1,3); 
 \draw [rounded corners=10] (3,0) -- (5,1) -- (1,3); 
 \draw [line width=4pt,white, rounded corners=10] (4,0) -- (5,0.5) -- (0,3);   
 \draw [rounded corners=10] (4,0) -- (5,0.5) -- (0,3);  
 \node at (0,-0.4) [mygreen,above] {${a_1}$};  
 \node at (1,-0.4) [mygreen,above] {${a_p}$}; 
 \node at (.5,-0.4) [mygreen,above] {${\cdots}$}; 
 \node at (3,-0.4) [mygreen,above] {${b_{1}}$}; 
 \node at (4,-0.4) [mygreen,above] {${b_q}$}; 
 \node at (3.5,-0.4) [mygreen,above] {${\cdots}$}; 
 \fill (4.85,2.5) circle (.08); 
 \fill (4.85,2) circle (.08);  
 \fill (4.85,1) circle (.08);  
 \fill (4.85,.5) circle (.08);  
 \node at (6.3,2.25) [myblue] {$J^{p;k}(\rho_{\Delta^{q;k}(\overline{b})}(\overline{a}))$}; 
 \node at (5.8,.75) [myblue] {$J^{q;k}(\overline{b})$}; 
 \node at (.5,3.5) [myred] {$\lambda_{\overline{a}}(\Delta^{q;k}(\overline{b}))$};  
 \node at (3.5,3.5) [myred] {$\Delta^{p;k}(\rho_{\Delta^{q;k}(\overline{b})}(\overline{a}))$};  
 \draw [line width=1.5pt,myred] (-.2,3) -- (-.2,3.2) -- (1.2,3.2) -- (1.2,3); 
 \draw [line width=1.5pt,myred] (2.8,3) -- (2.8,3.2) -- (4.2,3.2) -- (4.2,3); 
 \draw [line width=1.5pt,myblue] (5,2.7) -- (5.2,2.7) -- (5.2,1.9) -- (5,1.9);  
 \draw [line width=1.5pt,myblue] (5,1.2) -- (5.2,1.2) -- (5.2,.2) -- (5,.2); 
 \draw [line width=1.5pt,myred,dotted] (3.5,.75) -- (4.7,1.35);  
 \node at (3.4,.6) [myred] {$\scriptstyle \Delta^{q;k}(\overline{b})$};   
 \draw [line width=1.5pt,myred,dotted] (3.7,2.25) -- (4.9,1.65);  
 \node at (3.3,2.4) [myred] {$\scriptstyle \rho_{\Delta^{q;k}(\overline{b})}(\overline{a})$};   
\end{tikzpicture} 
   \caption{The $k$-Garside map $\Delta^{n;k}$ and the $k$-guitar map $J^{n;k}$ evaluated on products.}\label{P:GuitarProduct}
\end{figure}
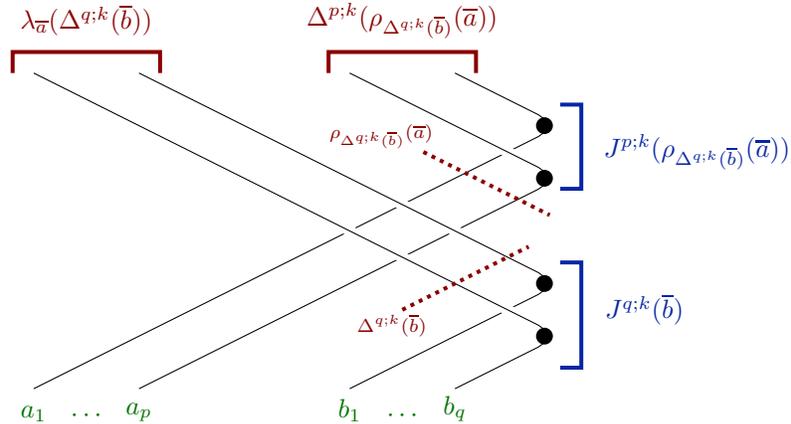
\end{proof}

\begin{rem} In the classical case $k=\Id$, the YBE implies $\lambda_{\Delta^{n;\Id}(\overline{c})}=\lambda_{\overline{c}}$ and $\rho_{\Delta^{n;\Id}(\overline{c})}=\rho_{\overline{c}}$ for all $\overline{c} \in X^n$, so \eqref{E:Jproduct} yields
\begin{align*}
J^{p+q;k}(\overline{a}\overline{b}) &= J^{p;k}(\rho_{\overline{b}}(\overline{a}))\ J^{q;k}(\overline{b}).
\end{align*}
To turn it into a $1$-cocycle condition, one needs to rewrite $J(\rho_{\overline{b}}(\overline{a}))$ as $J(\overline{a})\leftharpoonup \overline{b}$ for some right action of $M(X,r)$ on $M(X,r^{(\Id)})$. We omitted the superscripts for the map $J$ for simplicity. This action thus has to be defined by 
\[\overline{a}\leftharpoonup \overline{b} = J(\rho_{\overline{b}}(J^{-1}(\overline{a}))).\]
One easily checks that it is an action of $M(X,r)$ on $M(X,r^{(\Id)})$ by monoid morphisms. If one naively transposes this procedure to $k$-guitar maps, the operation obtained is no longer a monoid action.
\end{rem}

We have seen that the structure shelves of a solution and its generalised derived solutions coincide. This renders the structure shelves uninteresting invariants of generalised derived solutions. The $k$-versions of structure shelves (which need not to be shelves) are on the contrary rather useful.

\begin{defn}\label{D:kops}
Let $(X,r)$ be an RND solution, and $k$ its reflection. Define the following binary operation on $X$:
\[a \kop b =  \rho_{k(b)} \lambda_{\rho_a^{-1}(b)}(a).\]
If $(X,r)$ is LND instead, define the following binary operation on $X$:
\[b \klop a = \lambda_b k\rho_{\lambda_a^{-1}(b)}(a).\]
\end{defn}

Graphical versions of these operations in Fig.~\ref{P:kStrShelf} are very intuitive. When $k=\Id$, one recovers the right and the left structure shelf operations respectively.

\begin{figure}[h]
\centering
\begin{tikzpicture}[xscale=1,yscale=1]
\draw [rounded corners](0,0)--(0,0.25)--(0.4,0.4);
\draw [rounded corners](0.6,0.6)--(1,0.75)--(1,1.25)--(0,1.75)--(0,2);
\draw [rounded corners](1,0)--(1,0.25)--(0,0.75)--(0,1.25)--(0.4,1.4);
\draw [rounded corners] (0.6,1.6)--(1,1.75)--(1,2);
\fill (1,1) circle (.1);
\node  [mygreen,right] at (.95,.1){\large $a$};
\node  [mygreen,right] at (.95,.8){\large $b$};
\node  [right] at (.95,1.8){$a \kop b$};
\node  at (4,1){ };
\end{tikzpicture}
\begin{tikzpicture}[xscale=1,yscale=1]
\draw [rounded corners](0,0)--(0,0.25)--(0.4,0.4);
\draw [rounded corners](0.6,0.6)--(1,0.75)--(1,1.25)--(0,1.75)--(0,2);
\draw [rounded corners](1,0)--(1,0.25)--(0,0.75)--(0,1.25)--(0.4,1.4);
\draw [rounded corners] (0.6,1.6)--(1,1.75)--(1,2);
\fill (1,1) circle (.1);
\node  [mygreen,left] at (.05,.1){\large $a$};
\node  [mygreen,left] at (.05,1){\large $b$};
\node  [left] at (.05,1.8){$a \klop b$};
\end{tikzpicture}
\caption{The $k$-versions of structure shelf operations.}\label{P:kStrShelf}
\end{figure}
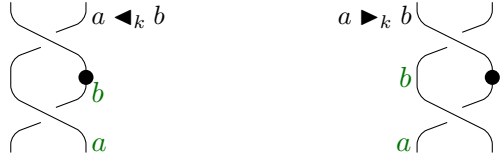

We will now study the operation $\kop$. The second operation $\klop$ will appear later in Remark \ref{R:VeryStrange}. Note that, contrary to the two structure shelf operations, these two operations play asymmetric roles. 

\begin{lem}
Let $(X,r)$ be an RND solution, and $k$ its reflection. For all $a,b,c \in X$, one has
\[(a \kop b) \kop c = (a \kop c') \kop b',\]
where $(b',c')=r^{(k)}(c,b)$.
\end{lem}

\begin{proof}
A graphical proof is given on Fig.~\ref{P:Action}.
\begin{figure}
\begin{tikzpicture}[xscale=0.75,yscale=1,>=latex]
 \node at (4,3.3) [right] {$\scriptstyle{(a \kop b) \kop c}$}; 
 \draw (4,2.5) -- (4,3.5); 
 \draw [line width=4pt,white, rounded corners=10] (3,0) -- (5,1) -- (3,2);
 \draw [rounded corners=10] (3,0) -- (5,1) -- (3,2);
 \draw [line width=4pt,white, rounded corners=10] (3,1.5) -- (5,2.5) -- (3,3.5); 
 \draw [rounded corners=10] (3,1.5) -- (5,2.5) -- (3,3.5); 
 \draw [line width=4pt,white]  (4.625,1.1875) -- (4.375,1.3125);    
 \draw (4.625,1.1875) -- (4.375,1.3125); 
 \draw [line width=4pt,white] (4,0) -- (4,2.5); 
 \draw (4,0) -- (4,2.5); 
 \draw [line width=4pt,white] (4.5,1.25) -- (3,2);
 \draw (4.5,1.25) -- (3,2);  
 \node at (4.6,2.3) [right] {${\color{mygreen} c}$};  
 \node at (4.6,.8) [right] {${\color{mygreen} b}$};   
 \node at (4,0.1) [right] {${\color{mygreen} a}$}; 
 \fill (4.75,2.5) circle (.08); 
 \fill (4.75,1) circle (.08);    
 \draw [latex->, rounded corners=10,myviolet] (5.5,1.5) -- (6,1.5)node[above] {$\scriptstyle{\text{YBE}}$} -- (6.5,1.5);
\end{tikzpicture}
\begin{tikzpicture}[xscale=0.75,yscale=1,>=latex]
 \draw (4,1.5) -- (4,3); 
 \draw [line width=4pt,white, rounded corners=10] (3,0) -- (5,1) -- (4,1.5) -- (5,2) -- (3,3);
 \draw [rounded corners=10] (3,0) -- (5,1) -- (4,1.5) -- (5,2) -- (3,3);
 \draw [line width=4pt,white, rounded corners=10] (2,0) -- (5,1.5) -- (2,3); 
 \draw [rounded corners=10] (2,0) -- (5,1.5) -- (2,3); 
 \draw [line width=4pt,white]  (4.625,1.1875) -- (4.375,1.3125);    
 \draw (4.625,1.1875) -- (4.375,1.3125); 
 \draw [line width=4pt,white] (4,0) -- (4,1.5); 
 \draw (4,0) -- (4,1.5); 
 \node at (4.6,1.3) [right] {${\color{mygreen} c}$};  
 \node at (4.6,.8) [right] {${\color{mygreen} b}$};   
 \node at (4,0.1) [right] {${\color{mygreen} a}$}; 
 \fill (4.75,1.5) circle (.08); 
 \fill (4.75,1) circle (.08);    
\end{tikzpicture}
\begin{tikzpicture}[xscale=0.75,yscale=1,>=latex]
 \draw [latex->, rounded corners=10,myviolet] (1.5,1.5) -- (2,1.5)node[above] {$\scriptstyle{\text{RE}}$} -- (2.5,1.5);
 \draw (4,1.5) -- (4,3); 
 \draw [line width=4pt,white, rounded corners=10] (3,0) -- (5,1) -- (4,1.5) -- (5,2) -- (3,3);
 \draw [rounded corners=10] (3,0) -- (5,1) -- (4,1.5) -- (5,2) -- (3,3);
 \draw [line width=4pt,white, rounded corners=10] (2,0) -- (5,1.5) -- (2,3); 
 \draw [rounded corners=10] (2,0) -- (5,1.5) -- (2,3); 
 \draw [line width=4pt,white]  (4.625,1.1875) -- (4.375,1.3125);    
 \draw (4.625,1.1875) -- (4.375,1.3125); 
 \draw [line width=4pt,white] (4,0) -- (4,1.5); 
 \draw (4,0) -- (4,1.5); 
 \node at (4.6,1.3) [right] {${\color{mygreen} c'}$};  
 \node at (4.6,1.8) [right] {${\color{mygreen} b'}$};   
 \node at (4,0.1) [right] {${\color{mygreen} a}$}; 
 \fill (4.75,1.5) circle (.08); 
 \fill (4.75,2) circle (.08);    
\end{tikzpicture}
\begin{tikzpicture}[xscale=0.75,yscale=1,>=latex]
 \draw [latex->, rounded corners=10,myviolet] (1.5,1.5) -- (2,1.5)node[above] {$\scriptstyle{\text{YBE}}$} -- (2.5,1.5);
 \draw (4,2.5) -- (4,3.5); 
 \draw [line width=4pt,white, rounded corners=10] (3,0) -- (5,1) -- (3,2);
 \draw [rounded corners=10] (3,0) -- (5,1) -- (3,2);
 \draw [line width=4pt,white, rounded corners=10] (3,1.5) -- (5,2.5) -- (3,3.5); 
 \draw [rounded corners=10] (3,1.5) -- (5,2.5) -- (3,3.5); 
 \draw [line width=4pt,white]  (4.625,1.1875) -- (4.375,1.3125);    
 \draw (4.625,1.1875) -- (4.375,1.3125); 
 \draw [line width=4pt,white] (4,0) -- (4,2.5); 
 \draw (4,0) -- (4,2.5); 
 \draw [line width=4pt,white] (4.5,1.25) -- (3,2);
 \draw (4.5,1.25) -- (3,2);  
 \node at (4.6,2.3) [right] {${\color{mygreen} b'}$};  
 \node at (4.6,.8) [right] {${\color{mygreen} c'}$};   
 \node at (4,0.1) [right] {${\color{mygreen} a}$}; 
 \fill (4.75,2.5) circle (.08); 
 \fill (4.75,1) circle (.08);    
 \node at (4,3.3) [right] {$\scriptstyle{(a \kop c') \kop b'}$}; 
\end{tikzpicture}
   \caption{The relation $(a \kop b) \kop c = (a \kop c') \kop b'$, where $(b',c')=r^{(k)}(c,b)$, is established by comparing the upper colors.}\label{P:Action}
\end{figure}
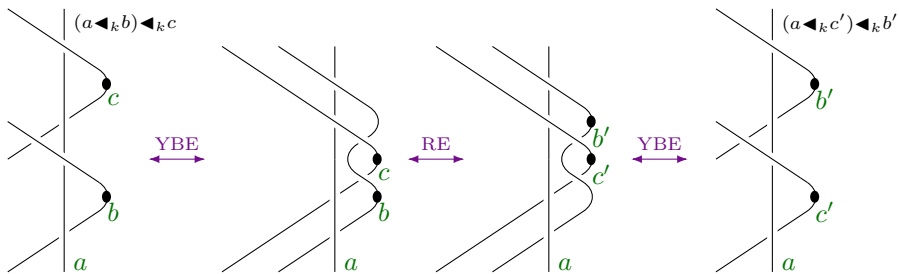
\end{proof}

This lemma directly implies
\begin{pro}\label{P:Actions}
Let $(X,r)$ be an RND solution, and $k$ its reflection. Then the right translations $ \kop b$ extend to a left action of the structure monoid $M(X,r^{(k)})$ on $X$.
\end{pro}


\section{Reflections for involutive solutions}

In this section $(X,r)$ is an involutive set-theoretic YBE solution. 

First, we will show that in this case in order to check the reflection equation~\eqref{E:RE} on $X^2$, it is often sufficient to look at one of the two coordinates of $X^2$ only.

To make this statement concrete, put
\[t(a,b) = \lambda_{\lambda_a(b)}k\rho_b (a), \qquad u(a,b) = \rho_{k \rho_b(a)}\lambda_a(b),\]
where $a,b \in X$. With these notations, \eqref{E:RE} can be rewritten as
\begin{align}
(t(a, k(b)),\ u(a, k(b))) &\ =\  (t(a, b),\ k(u(a, b))).\label{E:RE2}
\end{align}

\begin{thm}\label{T:InvolHalf}
\begin{enumerate}
\item\label{I:REinvol1} \cite{SVW18} Let $(X,r)$ be an LND involutive solution. A map $k \colon X \to X$ is a reflection for $(X,r)$ if and only if the left part 
\begin{align}
t(a, k(b)) &= t(a, b)\label{E:refl}
\end{align}
of~\eqref{E:RE2} holds for all $a,b \in X$.
\item\label{I:REinvol2} Let $(X,r)$ be an RND involutive solution. A map $k \colon X \to X$ is a reflection for $(X,r)$ if and only if the right part
\begin{align}
u(a, k(b)) &= k(u(a, b))\label{E:refl2}
\end{align}
of~\eqref{E:RE2} holds for all $a,b \in X$.
\end{enumerate}
\end{thm}

Relations \eqref{E:refl} and \eqref{E:refl2} can be explicitly written as follows:
\begin{align}
\lambda_{\lambda_a(b)}k\rho_b (a) &=\lambda_{\lambda_a(k(b))}k\rho_{k(b)} (a),\label{E:reflExplicit}\\
\rho_{k \rho_{k(b)}(a)}\lambda_ak(b) &= k\rho_{k \rho_b(a)}\lambda_a(b).\label{E:refl2Explicit}
\end{align}

\begin{proof}
We will prove only Point~\ref{I:REinvol2}. Point~\ref{I:REinvol1} can be shown using similar arguments; an alternative approach can be found in \cite{SVW18}.

Our aim is to deduce \eqref{E:refl} from \eqref{E:refl2}. Assume \eqref{E:refl2}. Take any $a,b \in X$. Put
\[(c,d)=rk_2(a,b), \qquad (e,f)=rk_2(c,d).\]
Since $r$ is involutive, this yields
\begin{align}
r(e,f)&=k_2(c,d)=(c,k(d)).\label{E:Invol1}
\end{align}
By \eqref{E:refl2}, we have
\[k_2rk_2r(a,b)=(e',f)\]
for some $e' \in X$. We need to prove that $e'=e$. Put 
\[(g,h)=r(e',f)=rk_2rk_2r(a,b).\]
By \eqref{E:refl2} again, $h$ is the same as the second component of
\[k_2rk_2rr(a,b)=k_2rk_2(a,b)=k_2(c,d)=(c,k(d)),\]
that is, $h=k(d)$. This yields
\[(g,k(d))=r(e',f).\]
Comparing with~\eqref{E:Invol1} and using the right non-degeneracy, one concludes that $e'=e$, as desired.
\end{proof}

This is used in the following construction of reflections:

\begin{thm}\label{T:InvolExamples}
\begin{enumerate}
\item\label{I:Ex1} \cite{SVW18} Let $(X,r)$ be an LND involutive solution. A map $k \colon X \to X$ commuting with all the $\lambda_a$'s, in the sense of
\begin{equation}\label{E:lambda_equivariant}
k\lambda_a = \lambda_a k \qquad \text{ for all } a \in X,
\end{equation}
is a reflection for $(X,r)$.
\item\label{I:Ex2} Let $(X,r)$ be an RND involutive solution. A map $k \colon X \to X$ satisfying
\begin{equation}\label{E:rho_commute}
\rho_{k(a)}=\rho_a \qquad \text{ for all } a \in X
\end{equation}
is a reflection for $(X,r)$.
\end{enumerate}
\end{thm}

\begin{proof}
\begin{enumerate}
\item According to Theorem~\ref{T:InvolHalf}, it is sufficient to check the relation \eqref{E:reflExplicit} for all $a,b \in X$. We have
\[\lambda_{\lambda_a(b)}k\rho_b (a)=k\lambda_{\lambda_a(b)}\rho_b (a)=k(a).\]
The first equality follows from \eqref{E:lambda_equivariant}, and the second one from the involutivity of~$r$. Similarly,
\[\lambda_{\lambda_a(k(b))}k\rho_{k(b)} (a)=k\lambda_{\lambda_a(k(b))}\rho_{k(b)} (a)=k(a).\]
\item According to Theorem~\ref{T:InvolHalf}, it is sufficient to check the relation \eqref{E:refl2Explicit} for all $a,b \in X$. We have
\[\rho_{k\rho_{k(b)}(a)}\lambda_ak(b)=\rho_{\rho_{k(b)}(a)}\lambda_ak(b)=k(b).\]
The first equality follows from \eqref{E:rho_commute}, and the second one from the involutivity of~$r$. Similarly,
\[k\rho_{k \rho_b(a)}\lambda_a(b)=k\rho_{\rho_b(a)}\lambda_a(b)=k(b).\qedhere\] 
\end{enumerate}
\end{proof}

\begin{exa}\label{EX:Perm}
Consider a permutation solution $r(a,b)=(f(b),f^{-1}(a))$, where $f$ is a permutation on~$X$. It is an involutive non-degenerate solution. Point~\ref{I:Ex1} of the theorem shows that any $k$ commuting with $f$ is a reflection, while Point~\ref{I:Ex2} asserts that all maps $k \colon X \to X$ are reflections. In this case, all the generalised derived solutions are simply the flips: $r^{(k)}(a,b)=(b,a)$.
\end{exa}

\begin{exa}\label{EX:3bis}
Let us resume Example~\ref{EX:3}. It is an involutive non-degenerate solution. Point~\ref{I:Ex1} of the theorem yields three reflections: $123$, $213$ and $333$. Point~\ref{I:Ex2} yields four reflections: all the maps respecting the decomposition $X=\{1,2\}\sqcup \{3\}$, that is, $123$, $213$, $113$, $223$. In total one gets five reflections, which is the complete list for this solution. For all these reflections except for $333$, we get $\rho_b=\rho_{k(b)}$ for all $b$, hence the $k$-derived solution is of the form $r^{(k)}(a,b)=(b',a)$. But it is involutive since the original solution is so, thus $r^{(k)}$ is a flip. The solution $r^{(333)}$ was described in Example~\ref{EX:3}. 
\end{exa}

\begin{pro}\label{P:InvolEqClasses}
Let $k$ be a reflection for an involutive solution $(X,r)$, and take two maps $\varphi,\psi \colon X \to X$ satisfying \eqref{E:lambda_equivariant}--\eqref{E:rho_commute}. Then the map $\varphi k \psi$ is a reflection.
\end{pro}

\begin{proof}
Relations \eqref{E:lambda_equivariant}--\eqref{E:rho_commute} can be assembled into 
$r\varphi _2=\varphi _1r$. But for involutive $r$, this implies $r\varphi _1=\varphi _2r$, since
\[\varphi _2r=rr\varphi _2r=r\varphi _1rr=r\varphi _1.\]
The same holds for $\psi$. This yields
\begin{align*}
r \varphi_2 k_2 \psi_2 r \varphi_2 k_2 \psi_2 &=(\varphi \times \varphi) rk_2rk_2 (\psi \times \psi) = (\varphi \times \varphi) k_2rk_2r (\psi \times \psi)\\
& =\varphi_2 k_2 \psi_2 r \varphi_2 k_2 \psi_2 r.\qedhere
\end{align*}
\end{proof}

This result suggests how to divide reflections for involutive solutions into equivalence classes.

\begin{exa}
For a trivial solution $r(a,b)=(b,a)$, all maps $X \to X$ are reflections, and they form a single equivalence class.
\end{exa}

\begin{exa}\label{EX:Permbis}
For a general permutation solution $r(a,b)=(f(b),f^{-1}(a))$ from Example \ref{EX:Perm}, all maps $X \to X$ are still reflections, but they may fall into several equivalence classes. For example, for $X=\{1,2\}$ and $f=(12)$, there are two maps satisfying \eqref{E:lambda_equivariant}--\eqref{E:rho_commute}: $\varphi=\Id$ and $\varphi=(12)$. Then the reflections fall into two classes: $\{12,21\}$ and $\{11,22\}$. On the other hand, if $f$ has a fixed point $p$, then any reflection can be postcomposed with the projection $\varphi(a)=p$, hence there is only one equivalence class.
\end{exa}

\begin{exa}
Let us resume Example~\ref{EX:3bis}. There are two maps satisfying \eqref{E:lambda_equivariant}--\eqref{E:rho_commute}:  $\varphi=\Id$ and $\varphi=(12)$. Pre- and post-composition with the latter divides the five solutions into three classes: $\{123,213\}$, $\{113,223\}$, and $\{333\}$.
\end{exa}

In the last two examples, the equivalence class turns out to be an invariant strictly refining the generalised derived solution. We will now show that in general two equivalent reflections yield isomorphic generalised derived solutions whenever the $\psi$'s used in the equivalence relation are bijections:

\begin{pro}\label{P:EquivReflVsDerivedSol}
Let $k$ be a reflection for an RND involutive solution $(X,r)$, and take two maps $\varphi,\psi \colon X \to X$ satisfying \eqref{E:lambda_equivariant}--\eqref{E:rho_commute}. One has
\[r^{(k)} (\psi \times \psi) = (\psi \times \psi) r^{(\varphi k \psi)}.\]
\end{pro}

\begin{proof}
In the proof of Propositoin~\ref{P:InvolEqClasses}, we showed that \eqref{E:lambda_equivariant}--\eqref{E:rho_commute} for $\psi$ imply $r\psi _1=\psi _2r$, hence $\psi \rho_a = \rho_a \psi$ and $\lambda_{\psi(a)}=\lambda_a$ for all $a \in X$. Now, recalling the definition of $r^{(k)}$, one has
\begin{align*}
r^{(k)} (\psi \times \psi) \colon (a,b) \longmapsto &(b'= \rho_{k(a')}\lambda_{\rho_{k\psi(b)}^{-1}\psi(a)}\psi(b)\ ,\ a'=  \rho_{\psi(b)}\rho_{k\psi(b)}^{-1}\psi(a))\\
&=(b' = \psi\rho_{k(a')}\lambda_{\psi\rho_{k\psi(b)}^{-1}(a)}(b)\ ,\ a'=  \psi\rho_{b}\rho_{k\psi(b)}^{-1}(a))\\
&=(b' = \psi\rho_{k(a')}\lambda_{\rho_{k\psi(b)}^{-1}(a)}(b)\ ,\ a'=  \psi\rho_{b}\rho_{k\psi(b)}^{-1}(a))\\
&= (b'= \psi \rho_{\varphi k \psi(a'')}\lambda_{\rho_{\varphi k \psi(b)}^{-1}(a)}(b)\ ,\ a'= \psi(a'')) \\
& \hspace*{3cm} \text{ where } a''= \rho_b\rho_{\varphi k \psi(b)}^{-1}(a)\\
&=(\psi \times \psi) r^{(\varphi k \psi)}(a,b). \qedhere
\end{align*}
\end{proof}


Finally, we unveil an unexpected relation between reflections and another algebraic structure.  Let $\Refl_{(X,r)}$ be the set of all reflections for $(X,r)$. Let $\Strange_X$ be the set of all \emph{strange} binary operations~$\ast$ on a set~$X$, that is, operations satisfying
\[(a \ast b) \ast a = b \ast a \qquad \text{ for all } a,b \in X.\]

\begin{thm}\label{T:Strange}
For any LND involutive YBE solution $(X,r)$, one has an injection
\begin{align*}
\Refl_{(X,r)} &\hookrightarrow \Strange_X,\\
k &\mapsto (b \klop a= \lambda_b k \lambda_b^{-1} (a)).
\end{align*}
Moreover, $k$ is a reflection if and only if the operation $\klop$ is strange.
\end{thm}

The operation $\klop$ above is a particular case of that from Definition~\ref{D:kops}.

Thus to check if $k$ is a reflection, it suffices to check the relation
\[\lambda_{\lambda_a k \lambda_a^{-1} (b)} k \lambda_{\lambda_a k \lambda_a^{-1} (b)}^{-1} (a) \ = \ \lambda_b k \lambda_b^{-1} (a).\]
Contrary to the alternative one-relation criteria \eqref{E:reflExplicit} and \eqref{E:refl2Explicit} from Theorem~\ref{T:InvolHalf}, it involves only the $\lambda$'s, and no $\rho$'s.

The remarkable feature of our injection is the independence of the set $\Strange_X$ from the map $r$. We identify reflections with strange operations $\ast$ such that the maps $\lambda_b^{-1}l_b\lambda_b$ do not depend on~$b$. Here $l_b$ is the left translation for $\ast$: $l_b(a)=b\ast a$. 

Using constraint programming, we computed the size of $\Strange_n$ for $n\leq 5$ (Table~\ref{Tab}). For $n=5$, the calculation took about 15 minutes. The size of $\Strange_6$ is expected to be huge.

\begin{table}[h]
\begin{tabular}{|c|c|c|c|c|c|}
\hline
$n$ & 1 & 2 & 3 & 4 & 5\tabularnewline
\hline 
$\vert Strange_{n} \vert$ & 1 & 4 & 44 & 4055 & 5589052\tabularnewline
\hline
\end{tabular}
\caption{Number of strange binary operations.}\label{Tab}
\end{table}

\begin{proof}[Proof of Theorem~\ref{T:Strange}]
The map $k$ can be uniquely reconstructed from $\klop$ (which in this proof we will write as $\ast$ for simplicity) and the $\lambda_b$'s, hence the injectivity. It remains to show that the proposed operation $\ast$ is indeed strange.

Due to the RE and the involutivity of $r$, one has the following equality of operators on $X^2$:
\[k_2rk_2r=rk_2rk_2=rk_2rk_2rr.\]
Let us now apply these operators to $(a,\lambda_a^{-1}(b))$. Since we are interested in what happens to the first component of $X^2$ only, we will write $\bullet$ in the second position. Applying $k_2rk_2r$, one obtains
\begin{align*}
(a,\lambda_a^{-1}(b)) &\overset{r}{\longmapsto} (b,\bullet) \overset{k_2}{\longmapsto} (b,\bullet) \overset{r}{\longmapsto} (b \ast a,\bullet)\overset{k_2}{\longmapsto} (b \ast a,\bullet).
\end{align*} 
Applying $rk_2rk_2rr$, one gets
\begin{align*}
(a,\lambda_a^{-1}(b)) &\overset{r}{\longmapsto} (b,\bullet) \overset{r}{\longmapsto} (a,\bullet)\overset{k_2}{\longmapsto} (a,\bullet) \overset{r}{\longmapsto} (a \ast b,\bullet) \overset{rk_2}{\longmapsto} ((a \ast b) \ast a,\bullet).
\end{align*} 
Hence $(a \ast b) \ast a = b \ast a$, as desired.

To prove the last statement, assume that $\klop$ is strange. The computations above show that the maps $k_2rk_2r$ and $rk_2rk_2=rk_2rk_2rr$ then have the same first component. But then, according to Theorem~\ref{T:InvolHalf}, these maps have to coincide.
\end{proof}

\begin{exa}\label{EX:3ter}
Let us resume Example~\ref{EX:3}. Recall that we have an involutive non-degenerate solution with $5$ reflections. The associated strange operations are
\[b \klop a = \begin{cases} k(a) & \text{ if } k= 123, 213 \text{ or } 333,\\
\lambda_b k(a) & \text{ if } k= 113 \text{ or } 223.\end{cases}\]
In other words, $b \klop a$ is always $k(a)$ except when $k= 113 \text{ or } 223$, $b=3$, and $a \neq 3$.
\end{exa}

\begin{rem}\label{R:VeryStrange}
The proof of Theorem~\ref{T:Strange} can be adapted to show that for any invertible LND solution, the reflection equation is equivalent to the following two conditions for all $a,b \in X$:
\[\begin{cases} (b \klop a) \klop b = (b \lop a) \klop b,\\
((b \klop a) \klop b) \lop  (b \klop a)= ((b \lop a) \klop b) \klop (b \lop a).
\end{cases}\]
Here $\lop$ is the left structure shelf operation for $(X,r)$.
Using notations
\begin{align*}
b \rightharpoonup a &= b \klop (b \wlop a) = \lambda_b k \lambda^{-1}_b (a), \\
 b \rightharpoondown a &= b \wlop (b \klop a) = \rho^{-1}_{\lambda_a^{-1}(b)} k\rho_{\lambda_a^{-1}(b)}(a),
\end{align*}
where $\wlop$ is the left-inverse of $\lop$, in the sense of
\[b \lop (b \wlop a) = b \wlop (b \lop a)=a,\]
the above system can be rewritten in a more compact form:
\[\begin{cases} (b \rightharpoonup a) \klop b = a \klop b,\\
(a \klop b) \rightharpoondown a = b \rightharpoonup a.
\end{cases}\]
For an involutive solution, the left structure shelf is trivial: $a \lop b = b$. Hence the operations $\rightharpoondown$, $\rightharpoonup$, and $\klop$ coincide, and both relations in the above system coincide with the strange axiom.
\end{rem}

\subsection*{Acknowledgments}
The authors are grateful to the reviewer for constructive suggestions and remarks. They are thankful to the organisers of the 2019 conference \textit{Groups, Rings and Associated Structures} in Spa, where the project was born.

\bibliographystyle{alpha}
\bibliography{refs}
\end{document}